\documentclass[11pt]{article}
\pdfoutput=1 

\usepackage[english]{babel}

\usepackage{geometry}
\usepackage{graphicx} \graphicspath{{./}{Figures/}}

\usepackage{amsmath,amsfonts,mathtools}
\numberwithin{equation}{section} 

\usepackage{enumitem}

\usepackage{siunitx}

\usepackage{amsthm}

\newtheoremstyle{italic}
{5pt}
{5pt}
{\itshape}
{}
{}
{}
{.5em}
{\bfseries{\thmname{#1}~\thmnumber{#2}.}\thmnote{~\textnormal{(#3)}}}

\newtheoremstyle{upright}
{5pt}
{5pt}
{\upshape}
{}
{\bfseries}
{}
{.5em}
{\bfseries{\thmname{#1}~\thmnumber{#2}.}\thmnote{~\textnormal{(\textit{#3}\textrm{)}}}}

\theoremstyle{italic}
\newtheorem{theorem}{Theorem}[section]
\newtheorem{lemma}[theorem]{Lemma}

\newtheorem{corollary}[theorem]{Corollary}

\theoremstyle{upright}

\newtheorem{remark}[theorem]{Remark}

\newtheorem{assumption}[theorem]{Assumption}
\newtheorem{example}[theorem]{Example}

\newtheorem{excont}{Example}[section]

\newcommand\blfootnote[1]{%
  \begingroup
  \renewcommand\thefootnote{}\footnote{#1}%
  \addtocounter{footnote}{-1}%
  \endgroup
}

\usepackage{dsfont}

\newcommand{\euler}{\mathrm{e}} 

\newcommand{\dinf}{\textnormal{d}} 

\newcommand{\Nat}{\mathbb{N}} 

\newcommand{\Prob}[1]{\mathbb{P}(#1)} 

\newcommand{\bld}[1]{\mathbf{#1}} 
\newcommand{\eb}[1]{\bld{e}_{#1}} 
\newcommand{\oneb}{\bld{1}} 

\newcommand{\la}{\lambda} 
\newcommand{\al}{\alpha}
\newcommand{\be}{\beta}
\newcommand{\ga}{\gamma}
\newcommand{\La}{\Lambda}

\title{Product-form solutions for a class of structured multi-dimensional Markov processes}
\author{Jori Selen\footnotemark[1] \footnotemark[2], Ivo J.B.F. Adan\footnotemark[1] \footnotemark[2], and Johan S.H. van Leeuwaarden\footnotemark[2]}

\begin{document}

\maketitle%
\renewcommand{\thefootnote}{\fnsymbol{footnote}}%
\footnotetext[1]{Department of Mechanical Engineering, Eindhoven University of Technology, The Netherlands}%
\footnotetext[2]{Department of Mathematics and Computer Science, Eindhoven University of Technology, The Netherlands}%
\renewcommand{\thefootnote}{\arabic{footnote}}%
\blfootnote{E-mail address: {\tt j.selen@tue.nl}}%
\renewcommand{\thefootnote}{\arabic{footnote}} \setcounter{footnote}{0}%

\begin{abstract}%
Motivated by queueing systems with heterogeneous parallel servers, we consider a class of structured multi-dimensional Markov processes whose state space can be partitioned into two parts: a finite set of boundary states and a structured multi-dimensional set of states, exactly one dimension of which is infinite. Using \textit{separation of variables}, we show that the equilibrium distribution, typically of the queue length, can be represented as a linear combination of product forms. For an important subclass of queueing systems, we characterize explicitly the waiting time distribution in terms of mixtures of exponentials.
\end{abstract}%


\section{Introduction}%
\label{sec:introduction}%

Motivated by queueing systems with heterogeneous parallel servers, we introduce a class of structured multi-dimensional Markov processes with one infinite dimension. While most research on multi-server queueing systems focuses on identical, or homogeneous servers, our class includes multi-server queueing systems with nonidentical, or heterogeneous servers.

Among the key examples of queueing systems included in this class are heterogeneous parallel servers with batch arrivals and service times consisting of two exponential phases (e.g.~Erlang-2 \cite{shapiro_ME2s1966}, Coxian-2 \cite{bertsimas_EkC2s1988}, hyperexponential-2 \cite{smit_numericHyper1983,smit_diffCustHyper1983}), multi-server queueing systems with service interruptions \cite{avi_serviceInterruptions1968}, setup times \cite{economou_setupTimes2005,adan_setupCosts2010} or batch service \cite{adan_batchService2000}, and a single exponential server fed by a superposition of independent, interrupted Poisson processes. Such queueing systems naturally arise in manufacturing, where jobs may arrive in batches, to be subsequently processed on multiple parallel batch machines. These machines can be nonidentical, differing in processing speed and processing variability and may require setup time after having been idle. For an overview of stochastic models, and queueing models in particular, for manufacturing systems we refer to \cite{buzacott_manufacturingSystems1993}.

The class of structured multi-dimensional Markov processes presented in this paper is an extension of the one in \cite{adan_zaragoza1999}. Whereas \cite{adan_zaragoza1999} deals with identical parallel servers, we allow for nonidentical parallel servers and batch arrivals. For this extended class we will determine the equilibrium distribution using a \textit{separation of variables} technique. This technique allows us to express the equilibrium distribution, typically of the queue length, as a linear combination of product forms (in terms of eigenvalues, and eigenvectors of the finite dimensions). Alternative methods to determine the equilibrium distribution for structured multi-dimensional Markov processes are the matrix-geometric method \cite{neuts_matrixGeometric1981} and the spectral expansion method \cite{mitrani_comparisonSpectralAndMatrixGeom1995}. The matrix-geometric method expresses the equilibrium distribution in terms of the rate matrix $R$, which is the minimal nonnegative solution of a nonlinear matrix equation. When $R$ is diagonalizable, the parameters of the product forms of the separation of variables technique are the eigenvalues and left eigenvectors of $R$. The spectral expansion method treats the equilibrium equations as a homogeneous vector difference equation with constant coefficients, and by substituting product forms, reduces to a single characteristic matrix polynomial for all eigenvalues.

By employing separation of variables, however, the single equation for all eigenvalues can be \textit{decomposed} into a system of equations for (much) fewer eigenvalues. Further, we can express the equilibrium distribution as a linear combination of product forms, the parameters of which can be obtained explicitly. Moreover, for an important subclass of queueing systems, we are able to obtain explicitly the waiting time distribution in terms of mixtures of exponentials.


\subsection{Motivating examples}%
\label{subsec:examples}%

We present three motivating examples that fall within the class of processes considered in this paper. The first two examples are extensions of examples in \cite{adan_zaragoza1999}. The third example serves as a running example throughout the paper.
\begin{example}[Parallel heterogeneous servers with Erlang-2 services] \label{ex:asymmetricK=1}
This example is an extension of the class of Markov processes presented in \cite{adan_zaragoza1999} by allowing different service behavior for each server. Jobs arrive according to a Poisson process with rate $\la$. Jobs are served FCFS by $c$ heterogeneous servers. The service time of server $i$ consists of two exponential phases, both with mean $1/\mu_i$.
\end{example}
\begin{example}[The $M^X/M/c$ queue with service interruptions] \label{ex:symmetricK>1}
This example considers the same class as in \cite{adan_zaragoza1999}, but now allowing for batch arrivals. Jobs arrive in batches of sizes that range from 1 to $K$. Batches of size $k$ arrive according to a Poisson process with rate $\la_k$. Jobs are served FCFS by $c$ identical parallel servers. Each server, when operative, serves jobs with rate $\mu$. However, servers are subject to breakdowns. The times that servers are operative are exponentially distributed with parameter $\theta$. The repair times are exponentially distributed with parameter $\nu$.
\end{example}
\begin{example}[Parallel heterogeneous servers with hypoexponential-2 service times and Poisson batch arrivals] \label{ex:running}
This example serves as a running example throughout the paper. We consider a queueing model that incorporates both extensions mentioned above. Jobs arrive in batches of sizes that range from 1 to $K$. Batches of size $k$ arrive according to a Poisson process with rate $\la_k$. Jobs are served FCFS by $c$ parallel heterogeneous servers. The service time of server $i$ consists of two exponential phases with mean $1/\mu_{1,i}$ and $1/\mu_{2,i}$, respectively.
\end{example}
%


\subsection{Description of the class of processes}%
\label{subsec:class_MP}%

We consider a class of irreducible multi-dimensional Markov processes on the state space $V \cup W$ with $V$ a finite set and $W$ defined by all states $\bld{n} = (n_0,n_1,\ldots,n_c)$ with $n_0 \in \Nat_0$ and $n_i \in \{0,1\}, ~ i = 1,\ldots,c$. The set $V$ serves to describe possibly less structured boundary behavior. For all three motivating examples of Section~\ref{subsec:examples}, the set $V$ consists of the states with one or more idle servers. In this paper we primarily focus on the analysis of the structured infinite set of states $W$. The determination of the probabilities in the finite set $V$ is of minor importance and not discussed here.

In the $n_0$-dimension, we allow for jumps in the negative direction of any size, whereas jumps in the positive direction are of maximum size $K$ with $K$ some finite positive integer. We call the states in $V$ and the states with $n_0 < K$ the \textit{boundary states}. For the states in $V$ we assume that no transitions to states $\bld{n}$ with $n_0 \ge K$ are possible, i.e., the only way to enter the states $\bld{n}$ with $n_0 \ge K$ from the set $V$ is via the states $\bld{n}$ with $n_0 < K$.

Transitions in the $(n_0,n_i)$-plane with $n_0 \ge 0$ are assumed to be independent of transitions in other $(n_0,n_j)$-planes, $j \neq i$. This allows us to describe in isolation the transition structure on this two-dimensional $(n_0,n_i)$-plane. More specifically, for the states with $n_0 \ge 0$, the following transitions in the $(n_0,n_i)$-plane are possible, with $k = -n_0,-n_0+1,\ldots,K$:
\begin{itemize}%
\item From $(n_0,0)$ to $(n_0+k,1)$ with rate $a_{k,i}$.
\item From $(n_0,0)$ to $(n_0+k,0)$ with rate $b_{k,i}$.
\item From $(n_0,1)$ to $(n_0+k,1)$ with rate $c_{k,i}$.
\item From $(n_0,1)$ to $(n_0+k,0)$ with rate $d_{k,i}$.
\item From $(n_0,n_i)$ to the states in $V$ with total rate $\sum_{k = -\infty}^{-n_0-1} \bigl( (1-n_i)(a_{k,i}+b_{k,i})+n_i(c_{k,i}+d_{k,i}) \bigr)$.
\end{itemize}%
Hence, transitions from $(n_0,\cdot)$ are only possible to states $(n,\cdot)$ with $n \le n_0 + K$. The transition rate diagram of the $(n_0,n_i)$-plane is shown in Figure~\ref{fig:transitions_n0_ni}.

\begin{figure}%
\centering%
\includegraphics{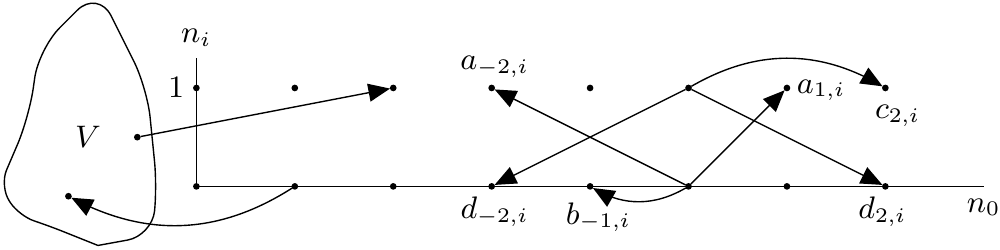}
\caption{Transition rate diagram of the process in the $(n_0,n_i)$-plane.}%
\label{fig:transitions_n0_ni}%
\end{figure}%

We assume that the Markov process is ergodic and denote by $p(\bld{n})$ the equilibrium probability of being in state $\bld{n}$. One of the main results obtained in this paper is the following expression for the equilibrium distribution in terms of a linear combination of $2^c K$ product forms:
\begin{equation}%
p(\bld{n}) = \sum_{j = 1}^{2^c K} \al_j \be_{0,j}^{n_0} \be_{1,j}^{n_1} \cdots \be_{c,j}^{n_c}, \quad \bld{n} \in W, \label{eqn:p(n)_intro}
\end{equation}%
where $\be_{i,j}$ and $\al_j$ are possibly complex-valued constants. We are able to establish \eqref{eqn:p(n)_intro} for the case $K = 1$, and also for $K \ge 1$ provided the Markov process is \textit{symmetric}, by which we mean that the transition rate diagram of each $(n_0,n_i)$-plane is identical. More specifically, by symmetry we mean that, for given $k$ all $a_{k,i}, ~ i = 1,\ldots,c$, are equal and similarly for the other rates. For $K \ge 2$ and nonsymmetric Markov processes (all Markov processes that are not symmetric), the proposed method might still work, as is supported by numerical evidence (not shown in this paper). However, proving that \eqref{eqn:p(n)_intro} is valid in this case is much more involved. In order to establish \eqref{eqn:p(n)_intro}, we employ a separation of variables technique, which is well known from the theory of partial differential equations \cite{garabedian_PDE1967} and has been applied before to an $E_k/E_r/c$ queueing model \cite{adan_EkErc1996}. Using this technique, a relation between $\be_{i,j}, ~ i = 1,\ldots,c$, and $\be_{0,j}$ is derived. Instrumental for this technique is to identify the constants $\be_{0,j}$ as the roots within the unit circle of a certain equation.

For symmetric Markov processes we can use an aggregated state space described by $(n_0,m)$ with $m = \sum_{i = 1}^c n_i$, and which forms a semi-infinite strip of states for which the equilibrium probabilities can be expressed as
\begin{equation}%
p(n_0,m) = \sum_{j = 1}^{K(c+1)} \ga_j \be_{0,j}^{n_0} \omega_j(m), \quad n_0 \in \Nat_0, ~ m = 0,\ldots,c, \label{eqn:p(n,m)_intro}
\end{equation}%
where $\ga_j$ are possibly complex-valued constants and $\omega_j$ is a row vector, the entries of which are given by $\omega_j(m) = \sum_{n_1 + \cdots + n_c = m} \be_{1,j}^{n_1} \cdots \be_{c,j}^{n_c}$ with $\be_{i,j}$ as in \eqref{eqn:p(n)_intro}.

In \cite{adan_zaragoza1999} the same class of symmetric Markov processes is studied for the special case $K = 1$, and a generating function technique is used to determine the equilibrium distribution. This generating function technique, however, does not seem applicable to the nonsymmetric processes studied in the this paper, which is why we use separation of variables instead.

We do use a generating function technique, together with \eqref{eqn:p(n,m)_intro} to describe the distribution of the first passage time to the set $V$, given that at time $t = 0$, the process starts in state $(n_0,m)$ and there are no transitions $a_k,\ldots,d_k$ with $k > 0$ for $t \ge 0$. For queueing systems, this first passage time is typically tantamount to the waiting time distribution (see e.g. \cite{adan_locking2000,bertsimas_FCFSWaitingTimeGGs1988}), which is also the case in this paper.


\subsection{Parameters of the motivating examples}%
\label{subsec:parameters_examples}%

We now briefly indicate how the motivating examples of Section~\ref{subsec:examples} fit in the class of Markov processes.
\begin{excont}[Continued]
Here, $n_0$ represents the number of jobs waiting in the queue and $n_i, ~ i = 1,\ldots,c$ is the number of finished phases at server $i$. We have $a_{0,i} = \mu_i$, $b_{1,i} = c_{1,i} = \la_{i}$ with $\sum_i \la_{i} = \la$, $d_{-1,i} = \mu_i$ and all other rates are equal to zero. Note that the rates $\la_{i} > 0$ can be chosen arbitrarily as long as they add up to $\la$.
\end{excont}
\begin{excont}[Continued]
In this case, $n_0$ represents the number of jobs waiting in the queue and $n_i, ~ i = 1,\ldots,c$ states whether server $i$ is operative. As the servers are identical we can omit the index $i$ to obtain the rates $a_0 = \nu$, $c_{-1} = \mu$, $b_k = c_k = \la_k/c, ~ (k = 1,\ldots,K)$, $d_0 = \theta$ and all other rates are equal to zero.
\end{excont}
\begin{excont}[Continued]
For this model, $n_0$ represents the number of jobs waiting in the queue and $n_i, ~ i = 1,\ldots,c$ is the number of finished phases at server $i$. We have $a_{0,i} = \mu_{1,i}$, $b_{k,i} = c_{k,i} = \la_{k,i} ~ (k = 1,\ldots,K)$ with $\sum_i \la_{k,i} = \la_k$ the arrival rate of batches of size $k$, $d_{-1,i}=\mu_{2,i}$ and all other rates are equal to zero.
\end{excont}
%


\subsection{Structure of the paper}%
\label{subsec:structure_paper}%

The remainder of the paper is organized as follows. In Section~\ref{sec:conditions} we introduce some further modeling assumptions and derive the ergodicity condition of the multi-dimensional Markov process. Section~\ref{sec:equilibrium_equations} concerns the analysis of the equilibrium equations to find the equilibrium distribution for jumps of maximum size $K = 1$. Symmetric processes with $K \ge 1$ are studied in Section~\ref{sec:symmetric_processes}. The aggregated state concept is introduced in Section~\ref{sec:aggregated_state_concept}. In Section~\ref{sec:absorption_times} we study the first passage time to the set $V$. Finally, in Section~\ref{sec:conclusion} we present some conclusions and directions for further research.


\section{Modeling assumptions}%
\label{sec:conditions}%

In this section we impose some conditions on the rates of the Markov process introduced in Section~\ref{sec:introduction} and we determine the ergodicity condition of the process. First, let us introduce the generating functions
\begin{equation}%
\begin{array}{ll}%
A_i(z) := \displaystyle \sum_{k = -\infty}^K a_{k,i} z^{K - k}, & B_i(z) := \displaystyle \sum_{k = -\infty}^K b_{k,i} z^{K - k}, \\
C_i(z) := \displaystyle \sum_{k = -\infty}^K c_{k,i} z^{K - k}, & D_i(z) := \displaystyle \sum_{k = -\infty}^K d_{k,i} z^{K - k}.
\end{array}%
\end{equation}%
We impose the following assumptions on the rates of the Markov process:
\begin{assumption}\label{as:rate}%
For $i = 1,2,\ldots,c$,
\begin{enumerate}[label = \textup{(\roman*)}]%
\item $A_i(1),B_i(1),C_i(1),D_i(1) < \infty$.
\item $A_i(1),D_i(1) > 0$.
\item $A_i'(1),B_i'(1),C_i'(1),D_i'(1)< \infty$.
\item $a_{K,i} = 0$ or $d_{K,i} = 0$.
\item $b_{K,i} = c_{K,i} \neq 0$.
\end{enumerate}%
\end{assumption}

The assumptions (i) and (iii) are intuitively clear: both the total outgoing rate and expected jump size should be finite. Assumption (ii) is imposed to exclude cases in which vertical transitions in the $(n_0,n_i)$-plane are only possible in one direction. Finally, assumptions (iv) and (v) are used to exclude exceptional cases. In Section~\ref{sec:equilibrium_equations} we show how these assumptions are used. Note that assumptions (iv) and (v) are satisfied for the motivating examples presented in Section~\ref{subsec:examples}.

We next present the ergodicity condition.
\begin{lemma}%
The Markov process is ergodic if and only if
\begin{align}%
0 &< \sum_{i = 1}^{c} \frac{1}{A_i(1) + D_i(1)} \Bigl( D_i(1)(A_i'(1) - KA_i(1) + B_i'(1) - KB_i(1)) \notag \\
&\quad + A_i(1)(C_i'(1) - KC_i(1) + D_i'(1) - KD_i(1)) \Bigr). \label{eqn:ergo_necessary}
\end{align}%
\end{lemma}%
\begin{proof}%
We require that the mean aggregate drift in the negative $n_0$-direction is larger than the mean aggregate drift in the positive $n_0$-direction; see Neuts' mean drift condition \cite{neuts_matrixGeometric1981}. This gives
\begin{equation}%
\sum_{i = 1}^{c} \sum_{k = -\infty}^{-1} k \pi_i \begin{pmatrix} b_{k,i} & a_{k,i} \\ d_{k,i} & c_{k,i} \end{pmatrix} \oneb > \sum_{i = 1}^{c} \sum_{k = 1}^K k \pi_i \begin{pmatrix} b_{k,i} & a_{k,i} \\ d_{k,i} & c_{k,i} \end{pmatrix} \oneb, \label{eqn:mean_aggregate_drift}
\end{equation}%
where $\oneb = \begin{pmatrix} 1 & 1 \end{pmatrix}^T$ and $\pi_i = \begin{pmatrix} \pi_i(0) & \pi_i(1) \end{pmatrix}$ is the equilibrium distribution of the $n_i$-dimension in the $(n_0,n_i)$-plane, which is the solution to
\begin{equation}%
\pi_i \begin{pmatrix} -A_i(1) & A_i(1) \\ D_i(1) & -D_i(1) \end{pmatrix} = 0.
\end{equation}%
Clearly, $\pi_i = \begin{pmatrix} \frac{D_i(1)}{A_i(1) + D_i(1)} & \frac{A_i(1)}{A_i(1) + D_i(1)} \end{pmatrix}$, and rewriting \eqref{eqn:mean_aggregate_drift} gives
\begin{equation}%
\sum_{i = 1}^{c} \frac{1}{A_i(1) + D_i(1)} \Bigl( D_i(1) \sum_{k = -\infty}^K k (a_{k,i} + b_{k,i}) + A_i(1) \sum_{k = -\infty}^K k(c_{k,i} + d_{k,i}) \Bigr) < 0.
\end{equation}%
Substituting $\sum_{k = -\infty}^K k a_{k,i} = K A_i(1) - A_i'(1)$ and similarly for the sums of $b_{k,i}$, $c_{k,i}$ and $d_{k,i}$ yields \eqref{eqn:ergo_necessary}.
\end{proof}%

We henceforth assume that condition~\eqref{eqn:ergo_necessary} holds.

\begin{remark}[Ergodicity condition for symmetric processes]%
For Markov process with identical rates in each $(n_0,n_i)$-plane, we can omit the index $i$ in \eqref{eqn:ergo_necessary}, so that \eqref{eqn:ergo_necessary} simplifies to
\begin{equation}%
0 < D(1)(A'(1) - KA(1) + B'(1) - KB(1)) + A(1)(C'(1) - KC(1) + D'(1) - KD(1)). \label{eqn:ergo_symm}
\end{equation}%
\end{remark}%


\section{Analysis of the equilibrium equations}%
\label{sec:equilibrium_equations}%

We now derive the equilibrium distribution $p(\bld{n})$ of the Markov process described in Section~\ref{subsec:class_MP}. We use a separation of variables technique and exploit the structure of the equilibrium equations for (non-boundary) states in the interior of the state space $W$.

Let $\eb{i}$ denote a row vector of zeros of length $c+1$ with a 1 at position $i$ with $0 \le i \le c$.  By equating the rate out of and the rate into state $\bld{n}$ we obtain
\begin{align}%
&\sum_{i = 1}^c \sum_{k = -\infty}^K \Bigl( (1 - n_i) (a_{k,i} + b_{k,i}) + n_i (c_{k,i} + d_{k,i}) \Bigr) p(\bld{n}) \notag \\
&= \sum_{i = 1}^c \sum_{k = -\infty}^K \Bigl( (1 - n_i)\bigl( b_{k,i} p(\bld{n} - k\eb{0}) + d_{k,i} p(\bld{n} - k\eb{0} + \eb{i}) \bigr) \notag \\
&\quad + n_i \bigl( a_{k,i} p(\bld{n} - k\eb{0} - \eb{i}) + c_{k,i} p(\bld{n} - k\eb{0}) \bigr) \Bigr), \label{eqn:EQ}
\end{align}%
which is valid for all states $\bld{n}$ with $n_0 \ge K$ and $n_i \in \{0,1\}, ~ i = 1,\ldots,c$. The equations \eqref{eqn:EQ} form the \textit{inner equations}, while the equilibrium equations for states with $n_0 < K$ and the set $V$ form the \textit{boundary equations}. As it turns out, the precise form of the boundary equations is not relevant for the first part of the analysis and these equations are therefore not presented. The boundary equations will be used to determine coefficients of the linear combination in Theorems~\ref{th:distK=1}, \ref{th:distSymm} and \ref{th:distAggregated}. Note that the number of boundary equations is equal to $|V| + 2^c K$.

We search for linearly independent candidate solutions $p(\bld{n})$ of \eqref{eqn:EQ} for which $\sum_{\bld{n}} p(\bld{n})$ absolutely converges. We have the following lemma.
\begin{lemma}\label{lem:p(n)}%
If there are $2^c K$ linearly independent and absolutely convergent solutions of \eqref{eqn:EQ}, labeled $p_j(\bld{n}), ~ j = 1,\ldots,2^c K$, then the equilibrium distribution $p(\bld{n})$ can be expressed as
\begin{equation}%
p(\bld{n}) = \sum_{j = 1}^{2^c K} \al_j p_j(\bld{n}), \quad \bld{n} \in W, \label{eqn:p(n)}
\end{equation}%
where the \textup{(}possibly complex-valued\textup{)} constants $\al_j$ are uniquely determined from the boundary equations and the normalization condition.
\end{lemma}%
\begin{proof}%
The proof is along similar lines as the proof of \cite[Theorem~4.1]{adan_EkErc1996}. For each choice of $\al_j$, the sequence $p(\bld{n})$ given by \eqref{eqn:p(n)} satisfies the equations \eqref{eqn:EQ}. The remaining equations are the boundary equations. These equations form a linear, homogeneous system for the unknown constants $\al_j$ and the boundary probabilities in $V$. The number of equations is equal to the number of unknowns, namely $|V| + 2^c K$. By first omitting the equilibrium equation in state $\bld{n} = \bld{0}$, say, the reduced homogeneous system of equilibrium equations obviously has a nonzero solution. The equilibrium equation in $\bld{n} = \bld{0}$ is automatically satisfied, since inserting $p(\bld{n})$ in the equations in all the other states and changing summations exactly yields the desired equation. Change of summations is allowed, since the sum of $p(\bld{n})$ over all states converges absolutely. It remains to show that $p(\bld{n})$ is a nonzero solution. This readily follows if at least one of the boundary probabilities is nonzero. If all these probabilities are zero, then at least one of the constants $\al_j$ must be nonzero. In this case, the property that the solutions $p_j(\bld{n}), ~ j = 1,\ldots,2^c K$ are linearly independent implies that $p(\bld{n})$ is a nonzero solution. Linear independence means that a linear combination of the solutions $\sum_{j = 1}^{2^c K} \al_j p_j(\bld{n})$ is equal to zero if and only if $\al_j = 0$ for all $j$. Using \cite[Theorem~1]{foster_ergodic1953}, we can conclude that the Markov process is ergodic and that a normalized version of $p(\bld{n})$ produces the equilibrium distribution. The uniqueness of the constants $\al_j$ follows from the uniqueness of the equilibrium distribution and the independence of the solutions $p_j(\bld{n}), ~ j = 1,\ldots,2^c K$.
\end{proof}%

From Lemma~\ref{lem:p(n)} and the uniqueness of the equilibrium distribution we obtain the following result.
\begin{corollary}\label{cor:dimensionSolutions}%
The dimension of the space of absolutely convergent solutions of \eqref{eqn:EQ} is at most $2^c K$.
\end{corollary}%
\begin{proof}%
Say we find $2^c K + 1$ linearly independent, absolutely convergent solutions. Let us label these solutions as $p_j(\bld{n}), ~ j = 1,\ldots,2^c K + 1$. According to Lemma~\ref{lem:p(n)} we can express $p(\bld{n})$ as a linear combination of the first $2^c K$ products. We know that at least one of the coefficients $\al_j$ has to be nonzero (since $p(\bld{n}) > 0$). Without loss of generality we assume that $\al_1 \neq 0$. Let us now construct $p(\bld{n})$ with the last $2^c K$ solutions and label the coefficients as $\xi_j$, thus ignoring the first product. As both sets of $2^c K$ linearly independent, absolutely convergent solutions represent the same equilibrium probability, taking the difference yields
\begin{equation}%
\al_1 p_1(\bld{n}) + \sum_{j = 2}^{2^c K}(\al_j - \xi_j)p_j(\bld{n}) - \xi_{2^c K + 1}~p_{2^c K + 1}(\bld{n}) = 0.
\end{equation}%
The solutions $p_j(\bld{n}), ~ j = 1,\ldots,2^c K + 1$ are linearly independent and therefore the only way to construct the zero solution is by settings all coefficients equal to 0. As $\al_1 \neq 0$, we get a contradiction, which implies that the corollary holds.
\end{proof}%

Lemma~\ref{lem:p(n)} reduces the problem of determining the equilibrium distribution to that of finding $2^c K$ linearly independent solutions of \eqref{eqn:EQ}. We next seek solutions satisfying the inner conditions \eqref{eqn:EQ} of the special form
\begin{equation}%
p(\bld{n}) = \be_0^{n_0} \be_1^{n_1} \cdots \be_c^{n_c}. \label{eqn:PF}
\end{equation}%
Clearly, only products that can be normalized, i.e.~for which the sum over all states converges absolutely, are useful. This implies that $|\be_0| < 1$.

Substituting \eqref{eqn:PF} into the equilibrium equations, dividing by $\be_0^{n_0 - K} \be_1^{n_1} \cdots \be_c^{n_c}$ and bringing all terms to one side, gives
\begin{align}%
0 &= \sum_{i = 1}^c \Bigl( (1 - n_i) \Bigl( B_i(\be_0) + D_i(\be_0)\be_i \Bigr) + n_i \Bigl( A_i(\be_0)\frac{1}{\be_i} + C_i(\be_0) \Bigr) \Bigr) \notag \\
&\quad - \be_0^K \sum_{i = 1}^{c} \bigl( (1 - n_i) (A_i(1) + B_i(1)) + n_i (C_i(1) + D_i(1)) \bigr). \label{eqn:EQ_GF}
\end{align}%
Since the above equation holds for all states $\bld{n}$ with $n_0 \ge K$ and the right-hand side of this equation must be zero, we require that the coefficients of $n_i$ vanish, which amounts to
\begin{equation}%
0 = A_i(\be_0)\frac{1}{\be_i} - B_i(\be_0) + C_i(\be_0) - D_i(\be_0)\be_i + \be_0^K(A_i(1) + B_i(1) - C_i(1) - D_i(1)). \label{eqn:coeff_ni}
\end{equation}%
This is a quadratic equation in $\be_i$, with solutions
\begin{equation}%
\be_i = \frac{F_i(\be_0)}{2D_i(\be_0)} + x_i \frac{\sqrt{ F_i(\be_0)^2 + 4A_i(\be_0)D_i(\be_0) }}{2D_i(\be_0)}, \quad i = 1,\ldots,c, \label{eqn:beta_i}
\end{equation}%
where $F_i(\be_0) = \be_0^K(A_i(1) + B_i(1) - C_i(1) - D_i(1)) - B_i(\be_0) + C_i(\be_0)$ and $x_i^2 = 1$, so that $x_i$ is either $-1$ or 1. As we cannot divide by zero, we need $D_i(\be_0) > 0$, which is always valid if $\be_0 \in (0,1)$ by Assumption~\ref{as:rate}(ii). If $\be_0 \in (0,1)$, then $F_i(\be_0)^2 + 4A_i(\be_0)D_i(\be_0) > 0$ and \eqref{eqn:beta_i} depends on $x_i$. Note that $\be_0$ is still to be determined.

We can now substitute \eqref{eqn:beta_i} into \eqref{eqn:EQ_GF} to find that
\begin{align}%
0 = \sum_{i = 1}^c &\Bigl( x_i \sqrt{F_i(\be_0)^2 + 4A_i(\be_0)D_i(\be_0)} + (B_i(\be_0) + C_i(\be_0)) \notag \\
&\quad - \be_0^K (A_i(1) + B_i(1) + C_i(1) + D_i(1)) \Bigr). \label{eqn:solve_beta_0}
\end{align}%
This leads to $2^c$ equations for $\be_0$ due to all possible combinations of $x_i$'s.
\begin{lemma}\label{lem:numberOfRootsK=1}%
For each combination of $x_i$'s, and $K = 1$, equation \eqref{eqn:solve_beta_0} has at least one root $\be_0 \in (0,1)$.
\end{lemma}%
\begin{proof}%
Let $h_i(\be_0)$ denote the function inside the summation in \eqref{eqn:solve_beta_0}. Assumptions~\ref{as:rate}(iv) and (v) imply that $F_i(\be_0)^2 + 4A_i(\be_0)D_i(\be_0)$ is equal to zero for $\be_0 = 0$, so that
\begin{equation}%
h_i(0) = B_i(0) + C_i(0), \quad h_i(1) = (x_i - 1)(A_i(1) + D_i(1)).
\end{equation}%
Thus,
\begin{equation}%
\sum_{i = 1}^c h_i(0) > 0, \quad \sum_{i = 1}^c h_i(1) \le 0.
\end{equation}%
Note that $\sum_{i = 1}^c h_i(1) = 0$ only occurs for $x_i = 1, ~ i = 1,\ldots,c$, in which case the derivative $\sum_{i = 1}^c h_i'(1)$ is equal to the right-hand side of \eqref{eqn:ergo_necessary} and is positive. This proves that there is at least one zero in $(0,1)$. For all other combinations of $x_i$'s, we conclude there is at least one zero as well, since $\sum_{i = 1}^c h_i(\be_0)$ is continuous.
\end{proof}%

By combining Lemma~\ref{lem:numberOfRootsK=1} and \eqref{eqn:beta_i} we then find the product $\be_0^{n_0} \cdots \be_c^{n_c}$ for each feasible combination of $x_i$'s, so we obtain at least $2^c$ products. Let us select one root $\be_0$ obtained from each equation and label these products as $\be_{0,j}^{n_0} \cdots \be_{c,j}^{n_c}, ~ j = 1,\ldots,2^c$. In this way we have characterized the product forms satisfying the equilibrium equations \eqref{eqn:EQ}.
\begin{lemma}\label{lem:linearlyIndependentProductsK=1}%
The products $\be_{0,j}^{n_0} \cdots \be_{c,j}^{n_c}, ~ j = 1,\ldots,2^c$ are linearly independent on $\bld{n} \in W$.
\end{lemma}%
\begin{proof}%
The proof is based on a property of the Vandermonde matrix. Let the numbers $y_1,\ldots,y_M$ satisfy $y_i \neq y_j, ~ i \neq j$. Then
\begin{equation}
\sum_{j = 1}^{M} \phi_j y_j^n = 0, ~ n = 0,1,\ldots,M-1 \iff \phi_j = 0, ~ j = 1,\ldots,M. \label{eqn:VandermondeProperty}
\end{equation}
To prove this property, note that the left-hand side of the if and only if relation in \eqref{eqn:VandermondeProperty} can be written as $Y\Phi = 0$ with $\Phi = \begin{pmatrix} \phi_1 & \cdots & \phi_M \end{pmatrix}^T$ and
\begin{equation}%
Y = \begin{pmatrix}%
1 & 1 & \cdots & 1 \\
y_1 & y_2 & \cdots & y_M \\
y_1^2 & y_2^2 & \cdots & y_M^2 \\
\vdots & \vdots & & \vdots \\
y_1^{M-1} & y_2^{M-1} & \cdots & y_M^{M-1}
\end{pmatrix}.%
\end{equation}%
Since $Y$ is a Vandermonde matrix, its determinant is given by $\prod_{1 \le i < j \le M} (y_i - y_j)$, which is nonzero, since $y_i \neq y_j$, if $i \neq j$. Thus property \eqref{eqn:VandermondeProperty} holds.

We want to show that the following property holds to prove linear independence:
\begin{equation}%
\sum_{j = 1}^{2^c} \al_j \be_{0,j}^{n_0} \be_{1,j}^{n_1} \cdots \be_{c,j}^{n_c} = 0, ~ \bld{n} \in W \iff \al_j = 0, ~ j = 1,\ldots,2^c. \label{eqn:linearIndependenceProperty}
\end{equation}%
Suppose the number of different $\be_{0,j}$'s is $M$. For any (fixed) choice of $n_i, ~ i = 1,\ldots,c$, we then rewrite the left-hand side of \eqref{eqn:linearIndependenceProperty} as
\begin{equation}%
\sum_{i} \be_{0,i}^{n_0} \sum_{\be_{0,j} = \be_{0,i}} \al_j \be_{1,j}^{n_1} \cdots \be_{c,j}^{n_c} = 0, \quad n_0 = 0,1,\ldots,M-1.
\end{equation}%
Hence, by \eqref{eqn:VandermondeProperty},
\begin{equation}%
\sum_{\be_{0,j} = \be_{0,i}} \al_j \be_{1,j}^{n_1} \cdots \be_{c,j}^{n_c} = 0, \quad n_i \in \{0,1\}, ~ i = 1,\ldots,c.
\end{equation}%
For $n_1 \in \{0,1\}$ and any choice of $n_i, ~ i = 2,\ldots,c$, we rewrite the above equation as
\begin{align}%
& \sum_{\substack{\mathllap{\be_{0,j}} = \mathrlap{\be_{0,i}} \\ \mathllap{x_1} = \mathrlap{-1}}} \al_j \be_{2,j}^{n_2} \cdots \be_{c,j}^{n_c}  \Bigl( \frac{F_1(\be_{0,i})}{2 D_1(\be_{0,i})} - \frac{\sqrt{F_1(\be_{0,i})^2 + 4A_1(\be_{0,i})D_1(\be_{0,i})}}{2 D_1(\be_{0,i})} \Bigr)^{n_1} \notag \\
+ &\sum_{\substack{\mathllap{\be_{0,j}} = \mathrlap{\be_{0,i}} \\ \mathllap{x_1} = \mathrlap{1}}} \al_j \be_{2,j}^{n_2} \cdots \be_{c,j}^{n_c} \Bigl( \frac{F_1(\be_{0,i})}{2 D_1(\be_{0,i})} + \frac{\sqrt{F_1(\be_{0,i})^2+4A_1(\be_{0,i})D_1(\be_{0,i})}}{2 D_1(\be_{0,i})} \Bigr)^{n_1} = 0. \label{eqn:VandermondeStep}
\end{align}%
Therefore, again by \eqref{eqn:VandermondeProperty}, for $n_i \in \{0,1\}, ~ i = 2,\ldots,c$,
\begin{equation}%
\sum_{\substack{\mathllap{\be_{0,j}} = \mathrlap{\be_{0,i}} \\ \mathllap{x_1} = \mathrlap{-1}}} \al_j \be_{2,j}^{n_2} \cdots \be_{c,j}^{n_c} = 0, \quad \textup{and} \quad \sum_{\substack{\mathllap{\be_{0,j}} = \mathrlap{\be_{0,i}} \\ \mathllap{x_1} = \mathrlap{1}}} \al_j \be_{2,j}^{n_2} \cdots \be_{c,j}^{n_c} = 0. \label{eqn:VandermondeStep2}
\end{equation}%
Repeating the steps in \eqref{eqn:VandermondeStep} and \eqref{eqn:VandermondeStep2} for $n_2,n_3,\ldots,n_c$, we finally get $\al_j = 0$ for all $j$ in accordance with \eqref{eqn:linearIndependenceProperty} and thus we conclude that the products are linearly independent.
\end{proof}%

Suppose we could find an additional root $\be_0$ in one of the equations in \eqref{eqn:solve_beta_0}, leading to $2^c + 1$ roots. Then we could use the same strategy as for the proof of Lemma~\ref{lem:linearlyIndependentProductsK=1} to prove that these $2^c + 1$ products are linearly independent. From Corollary~\ref{cor:dimensionSolutions} we deduce that finding an additional root is not possible and we formulate the following corollary.
\begin{corollary}\label{cor:numberOfSolutionsSolveBeta0}%
For each combination of $x_i$'s, and $K = 1$, equation~\eqref{eqn:solve_beta_0} has exactly one root $\be_0 \in (0,1)$, leading to a total of $2^c$ roots.
\end{corollary}%

The following theorem states that the equilibrium distribution $p(\bld{n}), ~ \bld{n} \in W$ for the case $K = 1$ can be expressed as a linear combination of product forms.
\begin{theorem}\label{th:distK=1}%
For the multi-dimensional Markov processes with $K = 1$, the equilibrium distribution takes the form
\begin{equation}%
p(\bld{n}) = \sum_{j=1}^{2^c} \al_j \be_{0,j}^{n_0} \be_{1,j}^{n_1} \cdots \be_{c,j}^{n_c}, \quad \bld{n} \in W,
\end{equation}%
where $\be_{0,j}$ are the roots of \eqref{eqn:solve_beta_0} in the interval $(0,1)$ for each combination of $x_i$'s and $\be_{i,j}, ~ i = 1,\ldots,c$ are the corresponding constants found from \eqref{eqn:beta_i}. The real-valued constants $\al_j$ are determined from the boundary equations and the normalization condition.
\end{theorem}%
\begin{proof}%
Corollary~\ref{cor:numberOfSolutionsSolveBeta0} says that we find exactly $2^c$ solutions $\be_0$. Using \eqref{eqn:beta_i} we find $2^c$ products $\be_{0,j}^{n_0} \cdots \be_{c,j}^{n_c}, ~ j = 1,\ldots,2^c$. These products are absolutely convergent, as $|\be_0| < 1$, and linearly independent, according to Lemma~\ref{lem:linearlyIndependentProductsK=1}. Then, Lemma~\ref{lem:p(n)} shows that the $2^c$ linearly independent and absolutely convergent solutions can be used to express the equilibrium distribution as a linear combination of these solutions.
\end{proof}%
\begin{corollary}%
For the case $K = 1$, the maximal root $\be_{0,j}$, which governs the tail behavior of the distribution, is always attained for all $x_i = 1$.
\end{corollary}%
\begin{proof}%
This proof uses the notation of the proof of Lemma~\ref{lem:numberOfRootsK=1}. Let $\hat{h}_i(\be_0), ~ i = 1,\ldots,c$ be the functions with all $x_i = 1$. Let $\hat{\be}_0$ be the root of $\sum_{i = 1}^c \hat{h}_i(\be_0) = 0$. Then, for $x_1,\ldots,x_c$ with at least one $x_i = -1$, we have that $\sum_{i = 1}^{c} h_i(0) > 0$ and $\sum_{i = 1}^c h_i(\hat{\be}_0) < \sum_{i = 1}^c \hat{h}_i(\hat{\be}_0) = 0$. Hence, together with Corollary~\ref{cor:numberOfSolutionsSolveBeta0}, we conclude that the root of $\sum_{i = 1}^c h_i(\be_0) = 0$ is in $(0,\hat{\be}_0)$.
\end{proof}%

Let us show how to obtain the equilibrium distribution for an example.
\begin{example}%
Consider Example~\ref{ex:running} with $K = 1$ and 2 heterogeneous servers. There are four different combinations of $x_1$ and $x_2$. For each combination we can solve \eqref{eqn:solve_beta_0} to obtain one root $\be_0 \in (0,1)$. This leads to four roots $\be_0$. For each root and corresponding combination of $x_1$ and $x_2$ we obtain $\be_1$ and $\be_2$ from \eqref{eqn:beta_i} to finally find four product forms. The coefficients of the linear combination $\al_j$ can be found from the boundary equations and normalization condition. In doing so, we have uniquely determined the equilibrium distribution $p(\bld{n})$ as shown in Theorem~\ref{th:distK=1}.
\end{example}%

Processes with larger jumps in the positive $n_0$-direction ($K \ge 2$) are harder to study than their counterparts $K = 1$. The number of roots of \eqref{eqn:solve_beta_0} may not be the same for each combination of $x_i$'s as is the case for the Markov processes that are skip-free in the positive $n_0$-direction. The proposed method might still work for processes with $K \ge 2$, which is supported by numerical experiments (not shown here). However, proving that there are $2^c K$ roots is much more involved. We will show this in the next section, where we study \textit{symmetric} Markov processes with $K \ge 1$.


\section{Symmetric processes}%
\label{sec:symmetric_processes}%

Remember that we refer by symmetric processes to the processes described in Section~\ref{subsec:class_MP} with the additional assumption that the rates in the $(n_0,n_i)$-planes are identical. The queueing systems mentioned in Section~\ref{subsec:examples} give rise to symmetric processes when the servers are identical (homogeneous). We exploit this symmetry to prove that \eqref{eqn:solve_beta_0} has $2^c K$ roots with $|\be_0| < 1$. For symmetric processes we omit the indices $i$ of the rates $a_{k,i},\ldots,d_{k,i}$.

With $\eta_c := -\frac{1}{c} \sum_{i = 1}^c x_i \in [-1,1]$, equation \eqref{eqn:solve_beta_0} simplifies to
\begin{equation}%
\eta_c\sqrt{F(\be_0)^2+4A(\be_0)D(\be_0)} = B(\be_0) + C(\be_0) - \be_0^K(A(1) + B(1) + C(1) + D(1)). \label{eqn:solve_beta_0_symm}
\end{equation}%
Note that if there exists a root $\be_0$ of \eqref{eqn:solve_beta_0_symm} with $|\be_{0}| < 1$ such that
\begin{equation}%
0 = F(\be_{0})^2 + 4A(\be_{0})D(\be_{0}) \label{eqn:overdetermined},
\end{equation}%
then this root $\be_{0}$ is a solution of \eqref{eqn:solve_beta_0_symm} for all combinations of $x_i$'s. This is not desired, because if \eqref{eqn:overdetermined} holds for some $\be_0$, the corresponding $\be_{i}, ~ i = 1,\ldots,c$, found from \eqref{eqn:beta_i} do not depend on $x_i$ and thus we obtain \textit{dependent} solutions. We therefore make the following assumption.
\begin{assumption}\label{as:overdetermined}%
For all roots $\be_{0}$ of \eqref{eqn:solve_beta_0_symm} with $|\be_{0}| < 1$, assume that $F(\be_{0})^2 + 4A(\be_{0})D(\be_{0}) \neq 0$.
\end{assumption}%

Note that Assumption~\ref{as:overdetermined} is satisfied for the case $K = 1$, because $\be_{0} \in (0,1)$ and thus the right-hand side of \eqref{eqn:overdetermined} is strictly positive, due to Assumptions~\ref{as:rate}(iv) and (v). Using Assumption~\ref{as:overdetermined} we should mention that the linear independence of the product forms derived in Lemma~\ref{lem:linearlyIndependentProductsK=1} is also applicable for $K \ge 2$, where the number of products is now $2^c K$. In the following example we show that for certain models Assumption~\ref{as:overdetermined} always holds.
\begin{example}%
Consider Example~\ref{ex:running}, but now with $\mu := \mu_{1,i} = \mu_{2,i}$, i.e.~parallel identical servers with Erlang-2 service times. The right-hand side of \eqref{eqn:overdetermined} then equals $4\mu^2\be_0^{2K + 1}$, thus \eqref{eqn:overdetermined} does not hold for feasible values of $\be_0$ as $\be_0 = 0$ is not a solution of \eqref{eqn:solve_beta_0_symm}.
\end{example}%

However, we can construct situations for which Assumption~\ref{as:overdetermined} is violated. We do so in the following example.
\begin{example}\label{ex:violated}%
Consider again Example~\ref{ex:running}, now with 2 parallel identical servers, where batches of size two arrive with rate $\la$ and there are no other arrivals. The ergodicity condition for this system is given by $\la < \frac{\mu_1\mu_2}{\mu_1 + \mu_2}$. Denote a root of \eqref{eqn:overdetermined} by $\be_{0,a}$ and a root for which the right-hand side of \eqref{eqn:solve_beta_0_symm} vanishes by $\be_{0,b}$, so that
\begin{equation}%
\be_{0,a} = -\frac{(\mu_1 - \mu_2)^2}{4\mu_1\mu_2}, \quad \be_{0,b} = -\sqrt{\frac{\la}{\la + \mu_1 + \mu_2}}.
\end{equation}%
For $\la = 1$, $\mu_1 = 6$ and $\mu_2 = 2$, we find that $\be_{0,a} = \be_{0,b} = -1/3$ and the ergodicity condition holds. Consequently, we find $\be_0 = -1/3$ as a root of \eqref{eqn:solve_beta_0_symm} for all combinations of $x_i$'s, and by \eqref{eqn:beta_i} we find that the products corresponding to this root are identical and thus dependent. It is interesting to see what type of products we find from \eqref{eqn:solve_beta_0_symm} for these parameter settings. This is summarized in Table~\ref{tbl:violatedProducts}. We see that there are only six different products. This is problematic as we require $2^c K = 8$ solutions. A possible solution to obtain the required number of independent solutions, is to try and find products of the form $(a n_0 + b)\be_0^{n_0} \be_1^{n_1}\be_2^{n_2}$ with $\be_0 = -1/3$, and to select $a$ and $b$ values such that the inner conditions \eqref{eqn:EQ} are satisfied. We leave this for future research.
\begin{table}%
\centering%
\begin{tabular}{S[table-format=3.2]S[table-format=3.2]|S[table-format=3.2]S[table-format=3.2]S[table-format=3.2]}%
{$x_1$} & {$x_2$} & {$\be_0$} & {$\be_1$} & {$\be_2$} \\
\hline
1 & 1 & 0.74 & 3.77 & 3.77 \\
& & -0.33 & -3 & -3 \\
\hline
-1 & 1 & 0.33 & -1.24 & 7.24 \\
& & -0.33 & -3 & -3 \\
\hline
1 & -1 & 0.33 & 7.24 & -1.24 \\
& & -0.33 & -3 & -3 \\
\hline
-1 & -1 & -0.32 & -2.48 & -2.48 \\
& & 0.26 & -1.28 & -1.28 \\
& & -0.33 & -3 & -3 \\
\end{tabular}%
\caption{Solutions of \eqref{eqn:solve_beta_0_symm} and corresponding $\be_i$'s for Example~\ref{ex:violated}.}%
\label{tbl:violatedProducts}%
\end{table}%
\end{example}%

Recall that in \eqref{eqn:beta_i} we need $D(\be_0) \neq 0$. As the roots are no longer restricted to the interval $(0,1)$, but more generally should lie within the unit circle, it is possible that $D(\be_{0}) = 0$ for some root $\be_{0}$. We therefore make the following assumption.
\begin{assumption}\label{as:Dnot0}%
For all roots $\be_{0}$ of \eqref{eqn:solve_beta_0_symm} with $|\be_{0}| < 1$, assume that $D(\be_{0}) \neq 0$.
\end{assumption}%

We now take squares of both sides of \eqref{eqn:solve_beta_0_symm}, thus eliminating the square root, and we will prove, under the ergodicity condition \eqref{eqn:ergo_symm}, the following result.

\begin{lemma}\label{lem:numberOfRootsK>1}%
Let $\eta_c := -\frac{1}{c} \sum_{i=1}^c x_i$. For each combination of $x_i$'s, the equation
\begin{equation}%
\eta_c^2 \bigl( F(\be_0)^2 + 4A(\be_0)D(\be_0) \bigr) = \bigl( B(\be_0) + C(\be_0) - \be_0^K(A(1) + B(1) + C(1) + D(1)) \bigr)^2, \label{eqn:solve_beta_0_symm_squared}
\end{equation}%
has $2K$ solutions $\be_0$ in the unit circle. 
\end{lemma}%
\begin{proof}%
We bring all terms in \eqref{eqn:solve_beta_0_symm_squared} to one side to obtain
\begin{equation}%
\bigl( B(\be_0) + C(\be_0) - \be_0^K(A(1) + B(1) + C(1) + D(1)) \bigr)^2 - \eta_c^2 F(\be_0)^2 -4\eta_c^2A(\be_0)D(\be_0) = 0. \label{eqn:zeros}
\end{equation}%
To determine the number of roots of \eqref{eqn:zeros} within the unit circle, we will use Rouch\'{e}'s theorem (see e.g. Theorem 9.2.3 in \cite{hille_analytic1959}, or more recent work in \cite{adan_Rouche2006}) twice. For now, let us ignore the $-4\eta_c^2A(\be_0)D(\be_0)$ term.  This leaves us with a difference of two squares, which we factor into its two terms. After some rearranging, these two terms can be written as
\begin{equation}%
B(\be_0)(1 - \eta_c) + C(\be_0)(1 + \eta_c) -\be_0^K \bigl( (A(1) + B(1))(1 - \eta_c) + (C(1) + D(1))(1 + \eta_c) \bigr) \label{eqn:factor1}
\end{equation}%
and
\begin{equation}%
B(\be_0)(1 + \eta_c) + C(\be_0)(1 - \eta_c) -\be_0^K \bigl( (A(1) + B(1))(1 + \eta_c) + (C(1) + D(1))(1 - \eta_c) \bigr). \label{eqn:factor2}
\end{equation}%
Let us first focus on \eqref{eqn:factor1}. We define
\begin{align}%
f_1(z) &:= -z^K \bigl( (A(1) + B(1))(1 - \eta_c) + (C(1) + D(1))(1 + \eta_c) \bigr), \\
g_1(z) &:= B(z)(1 - \eta_c) + C(z)(1 + \eta_c).
\end{align}%
Let $L \subset \mathbb{C}$ denote the closed unit disc and $\partial L$ the unit circle. Clearly, $f_1(z)$ has only one root $z = 0$ in $L$ of multiplicity $K$. We shall now show that
\begin{equation}%
|f_1(z)| > |g_1(z)|, \quad z \in \partial L,
\end{equation}%
so that it follows from Rouch\'{e}'s theorem that $f_1(z) + g_1(z)$ also has $K$ zeros in $L$.

Using that $(1 - \eta_c)$ and $(1 + \eta_c)$ are both nonnegative and evaluating $f_1(z)$ and $g_1(z)$ along $\partial L$, gives
\begin{align}%
|f_1(z)| &= | -z^K \bigl( (A(1) + B(1))(1 - \eta_c) + (C(1) + D(1))(1 + \eta_c) \bigr)| \notag \\
&= (A(1) + B(1))(1 - \eta_c) + (C(1) + D(1))(1 + \eta_c) \notag \\
&> B(1)(1 - \eta_c) + C(1)(1 + \eta_c) = B(|z|)(1 - \eta_c) + C(|z|)(1 + \eta_c) \notag \\
&\ge |B(z)(1 - \eta_c) + C(z)(1 + \eta_c)| = |g_1(z)|.
\end{align}%
Thus, \eqref{eqn:factor1} has $K$ roots in $L$ and using a similar argument, \eqref{eqn:factor2} also has $K$ roots in $L$. Therefore, the product of \eqref{eqn:factor1} and \eqref{eqn:factor2} yields $2K$ solutions within the unit circle.

Now we wish to once again use Rouch\'{e}'s theorem, this time on the complete equation \eqref{eqn:zeros}. Let us define
\begin{align}%
f_2(z) &:= \bigl( B(z) + C(z) - z^K(A(1) + B(1) + C(1) + D(1)) \bigr)^2 - \eta_c^2 F(z)^2, \\
g_2(z) &:= 4 \eta_c^2 A(z) D(z),
\end{align}%
where $f_2(z)$ is the product of \eqref{eqn:factor1} and \eqref{eqn:factor2} and note that \eqref{eqn:zeros} can be expressed as $f_2(z) - g_2(z) = 0$. We have already determined that $f_2(z)$ has $2K$ zeros in $L$. We will now work towards the proof that each equation of \eqref{eqn:solve_beta_0_symm_squared} has $2K$ solutions. For readability, we introduce
\begin{align}%
\tau_1(\eta_c) &:= (A(1) + B(1))(1 - \eta_c) + (C(1) + D(1))(1 + \eta_c), \\
\tau_2(\eta_c) &:= (A(1) + B(1))(1 + \eta_c) + (C(1) + D(1))(1 - \eta_c).
\end{align}%
Let us first establish a relation between $|f_2(z)|$ and $f_2(|z|)$:
\begin{align}%
|f_2(z)| &= | \bigl( B(z) + C(z) - z^K(A(1) + B(1) + C(1) + D(1)) \bigr)^2 - \eta_c^2F(z)^2 | \notag \\
&= | B(z)(1 - \eta_c) + C(z)(1 + \eta_c) - z^K \tau_1(\eta_c) | \notag \\
&\quad \times |B(z)(1+\eta_c) + C(z)(1-\eta_c) -z^K \tau_2(\eta_c)| \notag \\
&= | -B(z)(1 - \eta_c) - C(z)(1 + \eta_c) + z^K \tau_1(\eta_c) | \notag \\
&\quad \times | -B(z)(1 + \eta_c) - C(z)(1 - \eta_c) + z^K \tau_2(\eta_c) | \notag \\
&\ge \bigl( -B(|z|)(1 - \eta_c) - C(|z|)(1 + \eta_c) + |z|^K \tau_1(\eta_c) \bigr) \notag \\
&\quad \times \bigl( -B(|z|)(1 + \eta_c) - C(|z|)(1 - \eta_c) + |z|^K \tau_2(\eta_c) \bigr) \notag \\
&= f_2(|z|).
\end{align}%
Also, because of the triangle inequality, $g_2(|z|) \ge |g_2(z)|$. It thus suffices to show that
\begin{equation}%
f_2(|z|) > g_2(|z|), \quad z \in \partial L,
\end{equation}%
to prove that \eqref{eqn:zeros} has $2K$ roots for each combination of $x_i$'s.

Let us now evaluate $f_2(|z|)$ along $\partial L$, so $|z| = 1$, and
\begin{align}%
f_2(1) &= (-A(1) - D(1))^2 - \eta_c^2(A(1) - D(1))^2 \notag \\
&= A(1)^2 + 2A(1)B(1) + B(1)^2 - \eta_c^2(A(1)^2 - 2A(1)B(1) + D(1)^2) \notag \\
&= (1 - \eta_c^2) (A(1)^2 + D(1)^2) + 2(1 + \eta_c^2)A(1)D(1) \notag \\
&\ge 4\eta_c^2A(1)D(1) = g_2(1).
\end{align}%
Equality between $f_2(1)$ and $g_2(1)$ occurs only when $\eta_c^2 = 1$. We can apply Rouch\'{e}'s theorem for all $\eta_c^2 < 1$.

In order to use the theorem when $\eta_c^2 = 1$, we will essentially evaluate $f_2(|z|)$ and $g_2(|z|)$ along the circle $|z| = 1 - \epsilon$. To accomplish this, we use the Taylor expansion, $f(1 - \epsilon) = f(1) - \epsilon f'(1) + o(\epsilon)$. We want to show that $f_2(1 - \epsilon) > g_2(1 - \epsilon)$ for $\epsilon$ sufficiently small. As we have already determined that $f_2(1) = g_2(1)$, this leaves us to prove that
\begin{equation}%
f_2'(1) < g_2'(1).
\end{equation}%
Using $\eta_c^2 = 1$,
\begin{equation}%
f_2'(1) = -4A(1) \bigl( C'(1) - K(C(1) + D(1)) \bigr) - 4D(1) \bigl( B'(1) - K(A(1) + B(1)) \bigr)
\end{equation}%
and
\begin{equation}
g_2'(1) = 4(A(1)D'(1) + D(1)A'(1)).
\end{equation}
From the ergodicity condition \eqref{eqn:ergo_symm} we then know that $f_2'(1) < g_2'(1)$. Therefore, Rouch\'{e}'s theorem is applicable and \eqref{eqn:zeros} has $2K$ roots for each combination of $x_i$'s. 
\end{proof}%

\begin{remark}[Duplicate roots]\label{rem:duplicateRoots}%
We still need to resolve the issue of having twice as many roots as needed, namely $2^{c + 1} K$ roots, which is a result of squaring \eqref{eqn:solve_beta_0_symm}. For every possible value of $\eta_c$, say $\eta^*_c$, $-\eta^*_c$ is also one of the values of $\eta_c$. Using $\eta^*_c$ or $-\eta^*_c$ will yield the same roots $\be_0$, thus half of our roots are duplicates. If $\eta_c = 0$ we find $2K$ roots of which half are duplicates as well. This leaves us with $2^c K$ useful roots within the unit circle.
\end{remark}%
\begin{remark}[Number of unique solutions]\label{rem:numberOfUniqueSolutions}%
The sum $\sum_{i = 1}^c x_i$ can take $c + 1$ unique values, and using Lemma~\ref{lem:numberOfRootsK>1} and Remark~\ref{rem:duplicateRoots} we see that there are $K(c + 1)$ unique roots $\be_0$. Notice, however, that this still gives $2^c K$ unique solutions in the form of $\be_{0,j}^{n_0} \cdots \be_{c,j}^{n_c}$, as different values of $x_i$ lead to different products, even though $\be_{0,j}$ is the same. Here we have used Assumption~\ref{as:overdetermined} to make sure that the coefficient of $x_i$ in \eqref{eqn:beta_i} is nonzero.
\end{remark}%

Let us now formulate the following theorem.

\begin{theorem}\label{th:distSymm}%
For the symmetric multi-dimensional Markov processes with $K \ge 1$, the equilibrium distribution takes the form
\begin{equation}%
p(\bld{n}) = \sum_{j = 1}^{2^c K} \al_j \be_{0,j}^{n_0} \be_{1,j}^{n_1} \cdots \be_{c,j}^{n_c}, \quad \bld{n} \in W,
\end{equation}%
where $\be_{0,j}$ are the roots of \eqref{eqn:solve_beta_0_symm} within the unit circle and $\be_{i,j}, ~ i = 1,\ldots,c$ are the corresponding constants found from \eqref{eqn:beta_i}. The \textup{(}possibly complex-valued\textup{)} constants $\al_j$ are determined from the boundary equations and the normalization condition.
\end{theorem}%
\begin{proof}%
The proof is similar to the proof of Theorem~\ref{th:distK=1}. Lemma~\ref{lem:numberOfRootsK>1} together with Remark~\ref{rem:duplicateRoots} describe that we find exactly $2^c K$ solutions $\be_0$. Using \eqref{eqn:beta_i} we find $2^c K$ products $\be_{0,j}^{n_0} \cdots \be_{c,j}^{n_c}, ~ j = 1,\ldots,2^c K$. These products are absolutely convergent, as $|\be_0| < 1$, and linearly independent, according to Lemma~\ref{lem:linearlyIndependentProductsK=1}. Then, Lemma~\ref{lem:p(n)} shows that the $2^c K$ linearly independent and absolutely convergent solutions can be used to express the equilibrium distribution as a linear combination of these solutions.
\end{proof}%


\section{Aggregated state concept}%
\label{sec:aggregated_state_concept}%

We now show how the results from the symmetric multi-dimensional Markov process are related to a Markov process on a semi-infinite strip.

For symmetric Markov processes, we can exploit the symmetry and transform the Markov process to an aggregated state space characterized by $(n_0,m)$, where $m = \sum_{i = 1}^c n_i$. This state space is a semi-infinite strip of states. The transition structure induced by the structure of the multi-dimensional Markov process is as follows. For $n_0 \ge K$ and $k = -n_0,-n_0 + 1,\ldots,K$:
\begin{itemize}%
\item From $(n_0,m)$ to $(n_0 + k,m + 1)$ with rate $a_k(c - m)$.
\item From $(n_0,m)$ to $(n_0 + k,m)$ with rate $b_k(c - m) + c_k m$.
\item From $(n_0,m)$ to $(n_0 + k,m - 1)$ with rate $d_k m$.
\item From $(n_0,m)$ to the states in $V$ with total rate $\sum_{k = -\infty}^{-n_0 - 1} \bigl( (a_k + b_k)(c - m) + (c_k + d_k)m \bigr)$.
\end{itemize}%
Transitions from $(n_0,m)$ are only possible to states $(n,l)$ with $n \le n_0 + K$ and $l \in \{m-1,m,m+1\}$. Further, the rates $a_k$ and $b_k$ are scaled by a factor $c - m$, and the rates $c_k$ and $d_k$ are scaled by a factor $m$. Observe the linear structure due to the symmetry in the transitions in the dimensions $n_i, ~ i = 1,\ldots,c$. The transition rate diagram is shown in Figure~\ref{fig:stateSpaceSymmetric}.

\begin{figure}%
\centering%
\includegraphics{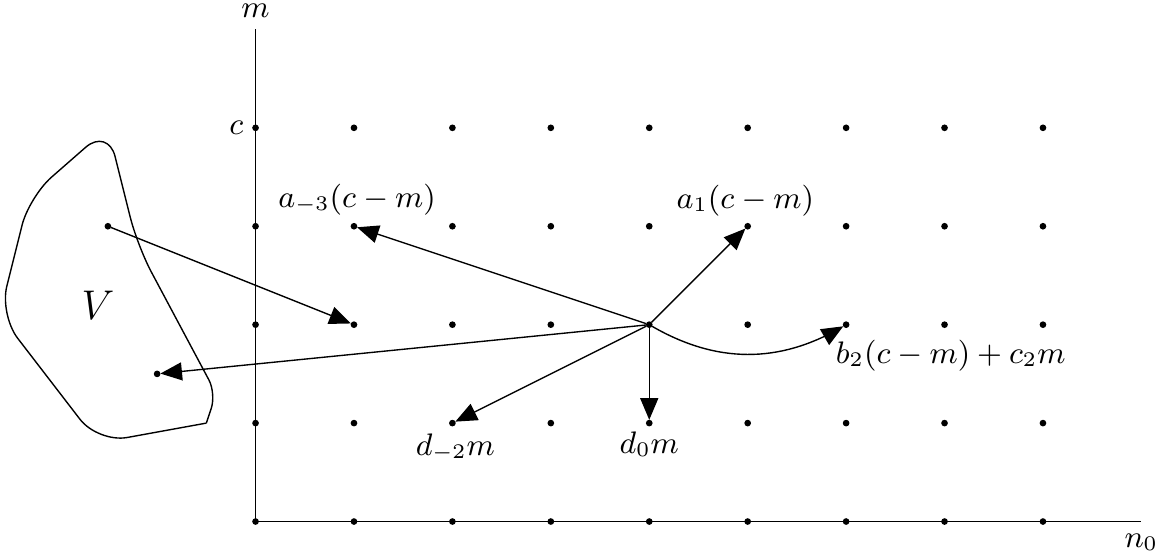}
\caption{Transition rate diagram for a symmetric Markov process on the aggregated state space.}%
\label{fig:stateSpaceSymmetric}%
\end{figure}%

Note that the ergodicity condition of the model on the aggregated state space is identical to the one derived for the symmetric multi-dimensional model, i.e.~\eqref{eqn:ergo_symm}.

We are interested in the equilibrium distribution of this aggregated state space. We obtain this distribution by using the basis of solutions of the multi-dimensional symmetric Markov process. Translation of the basis solutions $p_j(\bld{n}) = \be_{0,j}^{n_0} \cdots \be_{c,j}^{n_c}$ to the aggregated state description yields $p_j(n_0,m) = \be_{0,j}^{n_0} \omega_j(m)$, where $\omega_j(m)$ is given by
\begin{equation}%
\omega_j(m) = \sum_{n_1 + \cdots + n_c = m} \be_{1,j}^{n_1} \cdots \be_{c,j}^{n_c}. \label{eqn:omega_j(m)}
\end{equation}%
Basis solutions that are equal up to a permutation of the last $c$ factors are translated to exactly the same basis solutions in the new state space description. Hence, we relabel the basis solutions $\be_{0,j}^{n_0} \cdots \be_{c,j}^{n_c}, ~ j = 1,\ldots,2^c K$ such that the first $K(c + 1)$ solutions cannot be obtained from one another by permuting the last $c$ factors. Alternatively, one can only use the $c + 1$ unique values of $\eta_c$ in \eqref{eqn:solve_beta_0_symm} to obtain $K(c + 1)$ basis solutions. Hence, we obtain the reduced set of solutions $\be_{0,j}^{n_0} \omega_j(m), ~ j = 1,\ldots,K(c + 1)$. Then, Theorem~\ref{th:distAggregated} is a translation of Theorem~\ref{th:distSymm} in terms of the aggregated state space.
\begin{theorem}\label{th:distAggregated}%
For symmetric multi-dimensional Markov processes with the aggregated state space $(n_0,m)$, where $m = \sum_{i = 1}^c n_i$, the equilibrium distribution is of the form
\begin{equation}%
p(n_0,m) = \sum_{j = 1}^{K(c + 1)} \ga_j \be_{0,j}^{n_0} \omega_j(m), \quad n_0 \in \Nat_0, ~ m = 0,\ldots,c,
\end{equation}%
where $\be_{0,j}$ are the unique roots of \eqref{eqn:solve_beta_0_symm} within the unit circle and $\omega_j(m)$ is given by \eqref{eqn:omega_j(m)}. The \textup{(}possibly complex-valued\textup{)} constants $\ga_j$ are determined from the boundary equations and normalization condition.
\end{theorem}
\begin{remark}[Generating function technique]%
In \cite{adan_zaragoza1999}, a generating function technique is used to obtain the equilibrium distribution for the same symmetric class of processes with $K = 1$, where the generating function is over the finite index $m$. This technique can also be used to obtain the equilibrium distribution for the symmetric case $K \ge 2$, and in fact, leads to equations for $\be_0$ that are the same as \eqref{eqn:solve_beta_0_symm}.
\end{remark}%


\section{First passage times}%
\label{sec:absorption_times}%

In this section we study, for symmetric processes on the aggregated state space, the time until first passage to the set $V$, given that at time $t = 0$, the process starts in state $(n_0,m)$ and for $t \ge 0$ transitions in the positive $n_0$-direction are not possible (so no transitions with rate $a_k,\ldots,d_k, ~ k > 0$ for $t \ge 0$). The motivation for studying the first passage time is that for many queueing systems, this time is exactly the waiting time of a customer arriving in state $(n_0,m)$. Note that this equality holds due to the fact that we only consider the FCFS policy for these queueing systems.

Let $F_{n_0,m}(t)$ denote the probability that the first passage time to the set $V$ exceeds $t$, given that the process starts in state $(n_0,m)$ at time $t = 0$.
\begin{example}\label{ex:whyF}%
Consider Example~\ref{ex:running} with $c$ parallel identical servers and batch arrivals of maximum size $K$. Let $W$ denote the waiting time of a customer. By conditioning on the state seen on arrival and using that Poisson arrivals see time averages (PASTA, see \cite{wolff_PASTA1982}), we obtain
\begin{equation}%
\Prob{W > t} = \sum_{n_0 = 0}^\infty \sum_{m = 0}^c p(n_0,m) F_{n_0,m}(t) = \sum_{j = 1}^{K(c + 1)} \ga_j \sum_{n_0 = 0}^\infty \sum_{m = 0}^c \be_{0,j}^{n_0} \omega_j(m) F_{n_0,m}(t) ~ =: F(t). \label{eqn:F(t)def}
\end{equation}%
Hence, once $F_{n_0,m}(t)$ is known, we can determine the waiting time distribution.
\end{example}%

Let level $n_0$ denote the set of states $\{(n_0,0),\ldots,(n_0,c)\}$. We introduce a matrix notation $\La_i, ~ i = -n_0,-n_0 + 1,\ldots,K$ for the jumps from level $n_0$ to level $n_0 + i$, where $n_0 \ge 0$. To exemplify, the $(a,b)$-th entry of $\La_i$ describes the rate of going from state $(n_0,a)$ to state $(n_0 + i,b)$. By assuming a structure on the matrices $\La_i, ~ i = 1,\ldots,K$ we are able to explicitly derive the first passage time to the set $V$.
\begin{assumption}\label{as:AT_structure_matrices}%
Jumps in the positive $n_0$-direction are of the form
\begin{equation}%
\La_i = \la_i I, \quad i = 1,\ldots,K,
\end{equation}%
for some $\la_i \ge 0$ and $I$ the identity matrix.
\end{assumption}%

Note that for queueing systems, $\la_i I, ~ i = 1,\ldots,K$, models the Poisson arrival rate of a batch of size $i$. For the examples presented in Section~\ref{subsec:examples} the above assumption holds. Note that the assumption of Poisson arrivals is a natural assumption to make, otherwise \eqref{eqn:F(t)def} would not hold. By introducing a supplementary variable we are still able to analyze the waiting time distribution if there are no Poisson arrivals, see Remark~\ref{rem:IPP}.

Using the above assumption, we are now able to express $F(t)$ explicitly in terms of a finite sum.
\begin{theorem}\label{th:absorptionTime}%
The function $F(t)$ defined by the right-hand side of \eqref{eqn:F(t)def} is explicitly given by
\begin{equation}%
F(t) = \sum_{j = 1}^{K(c+1)} \frac{\ga_j \omega_j \oneb}{1 - \be_{0,j}} ~ \euler^{t \sum_{i = 1}^{K}(1 - \frac{1}{\be_{0,j}^{i}}) \la_i}, \quad t \ge 0, \label{eqn:AT}
\end{equation}%
with $\omega_j = \begin{pmatrix} \omega_j(0) & \omega_j(1) & \cdots & \omega_j(c) \end{pmatrix}$.
\end{theorem}%
\begin{proof}%
We first derive a set of differential equations for the probabilities $F_{n_0,m}(t)$. For small $\Delta t \ge 0$ it holds that
\begin{equation}%
F_{n_0}(t + \Delta t) = F_{n_0}(t) + \Delta t \sum_{i = 0}^{K} \La_i F_{n_0}(t) + \Delta t \sum_{k = 1}^{n_0} \La_{-k} F_{n_0 - k}(t) + o(\Delta t),
\end{equation}%
where $F_{n_0}(t) = \begin{pmatrix} F_{n_0,0}(t) & F_{n_0,1}(t) & \cdots & F_{n_0,c}(t) \end{pmatrix}^T$. Dividing these equations by $\Delta t$ and letting $\Delta t$ tend to zero yields the backward Kolmogorov equations
\begin{equation}%
F_{n_0}'(t) = \sum_{i = 0}^{K} \La_i F_{n_0}(t) + \sum_{k = 1}^{n_0} \La_{-k} F_{n_0 - k}(t), \quad n_0 \in \Nat_0 \label{eqn:Fdiff}
\end{equation}%
with initial condition $F_{n_0}(0) = \oneb$, where $\oneb$ is column vector of size $c + 1$ with all ones. To solve these differential equations we use Laplace transforms. Let
\begin{equation}%
F_{n_0}^*(s) = \int_{0}^\infty F_{n_0}(t) \euler^{-st} \, \dinf t, \quad s \ge 0.
\end{equation}%
Transforming the differential equations \eqref{eqn:Fdiff} for $F_{n_0}(t)$ 
gives
\begin{equation}%
\Bigl( sI - \sum_{i = 0}^{K} \La_i \Bigr)F_{n_0}^*(s) = \oneb + \sum_{k = 1}^{n_0} \La_{-k} F_{n_0 - k}^*(s). \label{eqn:AT_proof1}
\end{equation}%
From these equations, the Laplace transforms $F_{n_0}^*(s)$ can be solved recursively. Guided by Example~\ref{ex:whyF}, we define
\begin{equation}%
F(t) = \sum_{j = 1}^{K(c + 1)} \ga_j \omega_j \sum_{n_0 = 0}^\infty \be_{0,j}^{n_0} F_{n_0}(t), \quad t \ge 0,
\end{equation}%
and the Laplace transforms, for $s \ge 0$,
\begin{align}%
G_j^*(s) &= \sum_{n_0 = 0}^\infty \be_{0,j}^{n_0} F_{n_0}^*(s), \quad j = 1,\ldots,K(c + 1), \\
F^*(s) &= \int_{0}^\infty F(t) \euler^{-st} \,\dinf t = \sum_{j = 1}^{K(c + 1)} \ga_j \omega_j G_j^*(s).
\end{align}%
Multiplying \eqref{eqn:AT_proof1} by $\be_{0,j}^{n_0} \omega_j$ and summing over all $n_0$ leads to
\begin{align}%
\omega_j \Bigl( sI - \sum_{i = 0}^{K} \La_i \Bigr) G_{j}^*(s) &= \frac{\omega_j \oneb}{1 - \be_{0,j}} + \sum_{n_0 = 0}^{\infty} \sum_{k = 1}^{n_0} \be_{0,j}^{n_0} \omega_j \La_{-k} F_{n_0 - k}^*(s) \notag \\
&= \frac{\omega_j \oneb}{1 - \be_{0,j}} + \sum_{k = 1}^{\infty} \be_{0,j}^{k} \omega_j \La_{-k} \sum_{l = 0}^{\infty} \be_{0,j}^{l} F_{l}^*(s).
\end{align}%
Since
\begin{equation}%
0 = \sum_{n_0 = 0}^{\infty} \be_{0,j}^{n_0} \omega_j \La_{K - n_0} = \omega_j \sum_{n_0 = 0}^{K} \be_{0,j}^{n_0} \La_{K - n_0} + \be_{0,j}^K \omega_j \sum_{n_0 = 1}^{\infty} \be_{0,j}^{n_0} \La_{-n_0},
\end{equation}%
we get
\begin{equation}%
\omega_j \Bigl( sI - \sum_{i = 0}^{K} \La_i \Bigr)G_{j}^*(s) = \frac{\omega_j \oneb}{1 - \be_{0,j}} - \omega_j \sum_{i = 0}^{K} \frac{1}{\be_{0,j}^i} \La_i G_{j}^*(s),
\end{equation}%
and thus
\begin{equation}%
\omega_j \Bigl( sI - \sum_{i = 1}^{K} \Bigl(1 - \frac{1}{\be_{0,j}^i} \Bigr) \La_i \Bigr) G_{j}^*(s) = \frac{\omega_j \oneb}{1 - \be_{0,j}}.
\end{equation}%
We now use Assumption~\ref{as:AT_structure_matrices} to simplify the above expression as
\begin{equation}%
\omega_j G_{j}^*(s) = \frac{\omega_j \oneb}{1 - \be_{0,j}} \Bigl( s - \sum_{i = 1}^{K} \Bigl( 1 - \frac{1}{\be_{0,j}^i} \Bigr) \la_i \Bigr)^{-1}.
\end{equation}%
Multiplying by $\ga_j$ and summing over all $j$ finally yields
\begin{equation}%
F^*(s) = \sum_{j = 1}^{K(c + 1)} \frac{\ga_j \omega_j \oneb}{1 - \be_{0,j}} \Bigl( s - \sum_{i = 1}^{K} \Bigl( 1 - \frac{1}{\be_{0,j}^i} \Bigr) \la_i \Bigr)^{-1}.
\end{equation}%
The inverse of the Laplace transform is readily obtained as \eqref{eqn:AT}.
\end{proof}%

\begin{example}%
Again consider Example~\ref{ex:running} with parallel identical servers with $K = 2$. Let $\la_1$ be the arrival rate of single jobs and $\la_2$ the arrival rate of jobs in batches of size two. Since the first passage time to the set $V$ is equal to the waiting time, we have that $F(t) = \Prob{W > t}$. From Theorem~\ref{th:absorptionTime} we may immediately conclude that the waiting time distribution is the mixture of exponentials
\begin{equation}%
\Prob{W > t} = \sum_{j = 1}^{K(c + 1)} \frac{\ga_j \omega_j \oneb}{1 - \be_{0,j}} ~ \euler^{t \bigl( (1 - \frac{1}{\be_{0,j}})\la_1 + (1 - \frac{1}{\be_{0,j}^2}) \la_2 \bigr)}, \quad t \ge 0.
\end{equation}%
\end{example}%
\begin{remark}[$\sum IPP/M/1$ queuing systems]\label{rem:IPP}%
We consider a single exponential server fed by a superposition of independent, interrupted Poisson processes with rate $\la$, labeled $\sum IPP/M/1$. For an $\sum IPP/M/1$ queueing system Assumption~\ref{as:rate} does not hold. However, we can obtain the waiting time distribution explicitly. This suggest that the class of models with an explicit first passage time distribution in the form of mixtures of exponentials can possibly be extended beyond the class satisfying Assumption~\ref{as:rate}. Let us denote the probability that an arriving job sees the system in state $(n_0,m)$ by $a(n_0,m)$, where $m$ now reflects the number of `active' Poisson processes. For the $\sum IPP/M/1$ queuing system this is equal to
\begin{equation}%
a(n_0,m) = N p(n_0,m)m \la,
\end{equation}%
where $N$ is a normalizing constant, given by $N^{-1} = \sum_{n,k} p(n,k) k \la$. Hence, by Theorem~\ref{th:distAggregated}, we get that
\begin{equation}%
a(n_0,m) = N \sum_{j = 1}^{K(c + 1)} \ga_j \be_{0,j}^{n_0} \omega_j(m) m \la, \label{eqn:arrivalDist}
\end{equation}%
where
\begin{equation}%
N^{-1} = \sum_{j = 1}^{K(c + 1)} \sum_{k = 0}^c \frac{\omega_j(k) k \la}{1 - \be_{0,j}}.
\end{equation}%
The waiting time distribution can then be expressed as
\begin{equation}%
\Prob{W > t} = \sum_{n_0 = 1}^\infty \sum_{m = 0}^c a(n_0,m) F_{n_0,m}(t),
\end{equation}%
where $F_{n_0,m}(t)$ is an Erlang-$n_0$ distribution with parameter $\mu$. Note that for this specific model, $F_{n_0,m}(t)$ does not depend on $m$ and the set $V$ is level 0. Substituting expression \eqref{eqn:arrivalDist} for $a(n_0,m)$ and the expression for the Erlang-distribution then yields
\begin{equation}%
\Prob{W > t} = N \sum_{j = 1}^{K(c + 1)} \ga_j \sum_{m = 0}^c \omega_j(m) m \la \frac{\be_{0,j} \euler^{-(1 - \be_{0,j})\mu t}}{1 - \be_{0,j}},
\end{equation}%
which is a sum of exponentials, similar to \eqref{eqn:AT}.
\end{remark}%
%


\section{Conclusion}%
\label{sec:conclusion}%

We have introduced and analyzed a class of structured multi-dimensional Markov processes with one infinite dimension. By employing a separation of variables technique we showed that the equilibrium distribution can be represented as a linear combination of product forms (in terms of eigenvalues, and eigenvectors of the finite dimensions). We have presented an efficient approach to compute the product forms. Using the separation of variables technique, we were able to decompose the single equation for $2^c K$ eigenvalues that arises from the classical spectral expansion method to $2^c$ equations for fewer eigenvalues per equation, yet still obtaining $2^c K$ eigenvalues. The eigenvectors are then found from the corresponding eigenvalues. The weights of the linear combination were determined by a system of linear equations originating from the boundary equations and normalization condition. For symmetric processes we were able to aggregate the Markov process in terms of a process on a semi-infinite strip of states. The equilibrium distribution of this aggregated process can be expressed in terms of the parameters of the equilibrium distribution of the multi-dimensional Markov process. We have also shown that for these symmetric processes the first passage time to the set $V$ is a mixture of exponentials. For relevant queueing systems this gives the waiting time distribution.

There are at least two extensions of the class presented in this paper that seem worthwhile to explore. 
The first extension concerns the case of nonsymmetric processes with $K \ge 2$. It is presently unclear how to prove that one obtains $2^c K$ roots from $2^c$ equations. The other extension is to allow larger finite dimensions, i.e., letting $n_i \in \{0,1,\ldots,r\}, ~ i = 1,\ldots,c$ with $r > 1$. Due to tractability of the results it was opted not to do this for this paper. We believe that similar results can be obtained for these processes, however, one might need to move away from the power terms $\be_i^{n_i}$ of the finite dimensions, or impose additional structure, such as, for example, present in the $E_k/E_r/c$ queue \cite{adan_EkErc1996}.


\subsection*{Acknowledgement}%
\label{subsec:acknowledgement}%

This work was supported by a free competition grant from NWO and an ERC starting grant. The authors wish to thank the anonymous referees for their valuable comments on the presentation of the paper.


\bibliographystyle{plain}
\bibliography{ProductFormSolutionsClassMPBibliography}

\end{document}